\documentclass[12pt]{elsarticle}

\usepackage{amsfonts,amssymb,amsmath,amsthm}
\usepackage{cases}
\usepackage{color}

\makeatletter
\@addtoreset{equation}{section}
\def\theequation{\thesection .\arabic{equation}}
\renewcommand\theequation{
\thesection.\@arabic\c@equation
}
\makeatother
\newtheorem{lemme}{Lemma}
 \newtheorem{remarque}{Remark}
\newtheorem{theorem}{Theorem}

\newcommand{\phii}{\Phi}                  
\newcommand{\R}{\mathbb{R}}               
\newcommand{\G}{\textbf{g}}               
\newcommand{\n}{\textbf{n}}               
\newcommand{\N}{\mathbb{N}}               
\newcommand{\h}{\delta t}                 
\newcommand{\Ka}{{\bf K}}                 
                        
\newcommand{\V}{\textbf{V}}               
\newcommand{\sys}{\mathfrak{P}}           
\newcommand{\ga}{\gamma}                  
\newcommand{\Dlh}{D_l^h}                  
\newcommand{\pc}{\mathcal{P}_c}           

\newcommand{\ctea}{\mathcal{C}_1} 
\newcommand{\cteb}{\mathcal{C}_2} 
\newcommand{\X}{X} 
\newcommand{\Xlweta}{(X_l^w)^\eta} 
\newcommand{\Xlwetan}{(\X_l^w)^{n+1,\eta}} 
\newcommand{\sa}{s} 
\newcommand{\B}{\mathcal{B}}
\newcommand{\dt}{\partial_{t}}
\newcommand{\di}{\mathrm{div}}
\newcommand{\dd}{\mathrm{d}}
\newcommand{\grad}{\nabla}  

\begin{document}

\begin{frontmatter}


  \title{Study of degenerate parabolic system modeling the hydrogen displacement in a nuclear waste repository.}

  
  
  
 
   \author[florian]{Florian Caro}
  \ead{florian.caro@cea.fr} 
  \author[florian,mazen]{Bilal Saad\corref{florian,mazen}}
  \ead{bisaad@yahoo.fr} 
  
  \author[mazen]{Mazen  Saad} 
  
  \ead{Mazen.Saad@ec-nantes.fr}
  \address[florian]{CEA Saclay, DEN/DANS/DM2S/SFME/LSET, 
    91191 Gif Sur Yvette, France}
  \address[mazen]{Ecole Centrale de Nantes,
    Laboratoire de Math\'ematiques Jean Leray, UMR CNRS 6629\newline
    1, rue de la No\'e, 44321 Nantes France}
  
  \cortext[florian,mazen]{Corresponding author}





\begin{abstract}
  Our goal is the mathematical analysis of a two phase (liquid and
  gas) two components (water and hydrogen) system modeling the
  hydrogen displacement in a storage site for radioactive waste. We
  suppose that the water is only in the liquid phase and is
  incompressible. The hydrogen in the gas phase is supposed
  compressible and could be dissolved into the water with the Henry's
  law. The flow is described by the conservation of the mass of each
  components. The model is treated without simplified assumptions on
  the gas density. This model is degenerated due to vanishing terms.
  We establish an existence result for the nonlinear degenerate
  parabolic system based on new energy estimate on pressures.
  \end{abstract}

\begin{keyword}
  Degenerate system, nonlinear parabolic system, compressible flow,
  porous media\footnote{This work was partially supported by GNR
    MoMaS}.
\end{keyword}
\end{frontmatter}
\section{Introduction}

An important quantity of hydrogen can be produced by corrosion of
ferrous materials in a storage site for radioactive waste. A direct
consequence of this production is the growth of hydrogen pressure
around alveolus. This increasing gas pressure could break the
surrounding host rock and fractures could appear in the confinement
materials. This problem renews the mathematical interest in the
equation describing multiphase/multicomponent flows through porous
media. The cases of immiscible and incompressible flows have been
treated with "global pressure'' introduced by G. Chavent and J.
Jaffre \cite{chavent} by many authors, we refer for  instance to \cite{GM96, chen99,
  chen2002, feng} where existence results are obtained under various assumptions on physical data. For two
immiscible compressible flows without exchange between the phase, we
refer to \cite{CS08,CS07} where the authors obtain the existence of solution when the densities depend on the global
pressure and to \cite{ZS10,ZS11} for the general case where the
density of each phase depends on its own pressure. This approach is
also used in \cite{amaziane2,amaziane1} to treat a homogenization
problem of immiscible compressible water-gas flow in porous media.
For miscible and compressible flow, we refer to
\cite{choquet2008-2,choquet2008-1} for more details. 

In \cite{Bourgeat}, the authors derive a compositional model of
compressible multiphase flow in porous media. They focus their study
on models where the fluid is a mixture of two components: water
(mostly liquid) and hydrogen (mostly gas). An existence result has
been shown in \cite{smai} for this model under the assumptions of non
degeneracy and of strictly positive saturation.

Recently in
\cite{CS10}, the authors studied a new model of two compressible and
partially miscible phase flow in porous media, applied to gas
migration in an underground nuclear waste repository in the case where
the velocity of the mass exchange between dissolved hydrogen and
hydrogen in the gas phase is supposed finite.

Let us state the physical model used in this paper. We consider herein
a porous medium saturated with a fluid composed of two phases (liquid
and gas) and a mixture of two components (water and hydrogen) studied
in \cite{mikelic}. As reported in \cite{mikelic}, the author establish the existence of a weak solution, under non degeneracy and  slow
oscillation assumptions on the diagonal coefficients and
with small data for the hydrogen. Our aim is to show global solution for the degenerate system without restriction on data. \\
The water is supposed only in the liquid
phase (no vapor of water due to evaporation).  In order to define the
physical model, we write the \emph{mass conservation} of each
component
\begin{numcases}{(\sys)}
  \dt(\phii \sa_l\rho_l^{h}+\phii \sa_g\rho_g^{h} ) + \di(\rho_l^{h}
  \V_l + \rho_g^{h} \V_g) - \di(\rho_l\Dlh\nabla X_l^h)=
  r_g,\label{eq:hydrogene}\\
  \dt(\phii \sa_l\rho_l^{w} ) + \di(\rho_l^{w} \V_l ) =
  r_w.\label{eq:eau}
\end{numcases}
Here the subscript $l$ and $g$ represent respectively the liquid phase
and the gas phase.  Quantities $\phii$,  $\rho_l^{h}$, $\rho_g^{h}$, 
$\rho_\alpha=\rho_\alpha^{h}+ \rho_\alpha^{w}$, $\sa_\alpha$, $\X_l^h=\rho_l^h/\rho_l$ $\left(\X_l^h +\X_l^w = 1\right)$ and $\Dlh$
represent respectively the porosity of the medium, the density of dissolved hydrogen, the density of the hydrogen in the gas phase,  the density of the $\alpha$ phase ($\alpha = l,g$),  the saturation of the $\alpha$ phase ($\sa_l+\sa_g=1$),
the mass fraction of the hydrogen in the liquid
phase, the diffusion-dispersion tensor of the hydrogen in the liquid
phase. The velocity of each fluid $\V_\alpha$ is given by the Darcy's
law
\begin{align}
  \V_{\alpha}= -{\bf K} \frac{ k_{r_\alpha}(s_{\alpha})}{\mu_{\alpha}}
  \left(\nabla p_{\alpha}-\rho_{\alpha}(p_{\alpha})\textbf{g}\right),
\end{align}
where $\Ka$ is the intrinsic permeability tensor of the porous medium,
$k_{r_\alpha}$ the relative permeability of the $\alpha$ phase,
$\mu_\alpha$ the constant $\alpha$-phase's viscosity, $p_\alpha$ the
$\alpha$-phase's pressure and $\G$ the gravity. For detailed
presentation of the model we refer to the presentation of the
benchmark Couplex-Gaz  \cite{talandier} and \cite{Bourgeat,smai}.
The capillary pressure law is defined as 
$$
p_c(\sa_l) = p_g - p_l,
$$
is decreasing, $(\frac{\dd p_c}{\dd \sa_l}(\sa_l)<0, \text{ for all }
\sa_l\in[0,1])$ and $p_c(1)=0$.

The system \eqref{eq:hydrogene}--\eqref{eq:eau} is not complete, to
closing the system, we use the ideal gas law and the Henry's law
\begin{align}\label{law}
  \rho_g^h = \frac{M^h}{R T}p_g, \; \rho_l^h = M^h H^h p_g,
\end{align}
where the quantities $M^h$, $H^h$, $R$ and $T$ represent respectively
the molar mass of hydrogen, the Henry's constant for hydrogen, the
universal constant of perfect gases and $T$ the temperature.

By these formulation, the system \eqref{eq:hydrogene}--\eqref{eq:eau}
is closed and we choose the liquid and gas pressures as unknowns. From
\eqref{law}, the henry's law combined to the ideal gas law, to obtain
  that the density of hydrogen gas is proportional to the density of
  hydrogen dissolved
\begin{equation}\label{def:relation_densite}
  \rho_g^h = \ctea \rho_l^h \text{ where } \ctea = \frac{1}{H_h R T}(=52.51).
\end{equation}
Remark that the density of water $\rho_l^w$ in the liquid phase is
constant and from the Henry's law, we can write
$$
\rho_l\nabla \X_l^h = \cteb \X_l^w\nabla p_g,
$$
where $\cteb$ is a constant equal to $H^h M^h$.

Then the system \eqref{eq:hydrogene}--\eqref{eq:eau} can be writen as
\begin{numcases}{(\sys)}
  \partial_{t}\left(\phii m(\sa_{l}) \rho_l^h \right) +
  \di\left(\rho_l^h\V_{l} + \ctea \rho_l^h\V_{g}\right) -
  \di\left(\cteb \X_l^w \Dlh\nabla p_g \right)=r_g,&
\label{eq:hydrogene1}\\
  \partial_{t}\left(\phii \sa_{l}\right) +
  \di\left(\V_{l}\right)=\frac{r_w}{\rho_l^w}.&
  \label{eq:eau1}
\end{numcases} 
where $m(\sa_l)=\sa_l+\ctea \sa_g$.\\

Note that the mass exchange between dissolved hydrogen and
hydrogen in the gas phase is static using the Henry's law opposite then
supposed in \cite{CS10}.

\section{Assumptions and main result}\label{sec:assump}

The main point is to handle a priori estimates on the
approximate solution. Due to the degeneracy for dissipative terms
$\di(\rho_l^h M_{\alpha}\nabla p_\alpha )$, we can't control the
discrete gradient of pressure since the mobility of each phase
vanishes in the region where the phase is missing. So, we are going to
use the feature of global pressure to obtain uniform estimates on the
gradient of the global pressure and the gradient of a capillary term
to treat the degeneracy of the dissipative terms. Let summarize some
useful notations in the sequel.  We recall the conception of the
global pressure as describe in \cite{chavent}
\begin{align}\label{form}
M(s_l)\nabla p = M_l(s_l) \nabla p_l + M_g(s_g) \nabla p_g,
\end{align}
with the $\alpha$-phase's mobility $M_\alpha$ and the total mobility are defined by 
$$
M_{\alpha}(s_{\alpha})=k_{r_\alpha}(s_{\alpha})/ \mu_{\alpha}, \quad
M(s_l) = M_l(s_l)+M_g(s_g).
$$
This pressure $p$ can be written as
\begin{align}\label{p}
  p=p_{g}+\tilde{p}(s_l)=p_{l}+\bar{p}(s_l),
\end{align}
with 
\begin{align*}
  \frac{\dd \tilde{p} }{\dd s_l} = -\frac{M_{l}(s_{l})}{M(s_{l})}
  \frac{\dd p_c}{\dd s_l}  \text{ and } 
  \frac{\dd \bar{p}}{\dd s_l} = \frac{M_{g}(s_{g})}{M(s_{l})}
  \frac{\dd p_c}{\dd s_l}. 
\end{align*}
We also define the contribution of capillary terms by 
$$
\ga (s_l)=-\frac{M_{l}(s_l)M_{g}(s_g)}{M(s_l)}\frac{\dd p_c}{\dd
  s_l}(s_l)\geq 0 \text{ and } \mathcal{B}(s_l)=\int_{0}^{s_l}\ \ga(z)
\dd z.
$$
We complete the description of the model
\eqref{eq:hydrogene1}-\eqref{eq:eau1} by introducing boundary
conditions and initial conditions. Let $T>0$ be the final time fixed
and let be $\Omega$ a bounded open subset of $\R^d\ (d\geq1)$. We set
$Q_{T}=(0,T)\times \Omega$, $\Sigma_{T}=(0,T)\times
\partial \Omega$ and we note $\Gamma_l$ the part of the boundary of
$\Omega$ where the liquid saturation is imposed to one and
$\Gamma_n=\Gamma \backslash \Gamma_l$. The chosen mixed boundary
conditions on the pressures are
$$
\left\{ \begin{aligned} p_g(t,x)=p_l(t,x)=0 & \text{ on } (0,T)\times
    \Gamma_l, \\ \V_{l}\cdot \n = \V_{g}\cdot \n = 0 & \text{ on }
    (0,T)\times \Gamma_n,\\  \X_l^w \Dlh\nabla p_g\cdot \n =0 &
    \text{ on } (0,T)\times \Gamma_n,
\end{aligned} \right.
$$
where $\textbf{n}$ is the outward normal to $\Gamma_n$.

The initial conditions are defined on pressures
\begin{equation}\label{cd:initiale}
  p_{\alpha}(t=0) =p^{0}_{\alpha}\text{ in } \Omega,  \text{ for } \alpha=l,g.
\end{equation}
Next we introduce a classically physically relevant assumptions on the
coefficients of the system.
\begin{enumerate}[({H}1)]
\item The porosity $\phi\in W^{1,\infty}(\Omega)$ and there is two
  positive constants $\phi_{0}$ and $\phi_{1}$ such that $\phi_{0}\leq
  \phi(x)\leq \phi_{1}$ almost everywhere $x\in \Omega$.
\item There exists two positive constants $k_{0}$ and
  $k_{\infty}$ such that
  $$
  \left\| \textbf{K} \right\|_{(L^{\infty}(\Omega))^{d \times d}} \leq
  k_{\infty} \text{ and } \left<\textbf{K}(x) \xi,\xi \right>\geq k_{0}| \xi
  |^{2}, \forall \xi \in \R^d.
  $$
\item The functions $M_{l}$ and $M_{g}\in {\cal C}^{0}([0,1],\R^{+})$, $
  M_{\alpha}(s_{\alpha}=0)=0.$ and there is a positive
  constant $m_{0}>0$ such that for all $s_{l}\in [0,1]$,
  $$
  M_{l}(s_l) + M_{g}(s_g)\geq m_{0}.
  $$
\item The densities $\rho_{\alpha}$ $(\alpha\text{=l,g)}$ are in
  ${\cal C}^{1}(\R)$, increasing and there exists two positive constants
  $\rho_{m}>0$ and $\rho_{M}>0 $ such that $$0<\rho_{m}\leq
  \rho_{\alpha}(p_{\alpha})\leq \rho_{M},\ \alpha=l, g.$$
\item The capillary pressure fonction $p_c\in
  \mathcal{C}^{1}([0,1];\R^{+})$ and there exists $\underline{p_c}>0$
  such that $0<\underline{p_c}\leq |\frac{\dd p_c}{\dd s_l}|$.
\item The functions $\ r_{\omega}, \ r_{g}\in L^{2}(Q_{T}) \text{ and
  } r_{\omega}$, $r_{g}\geq 0$ a.e. for all $(t,x)\in Q_T$.
\item The diffusion-dispersion tensor $\Dlh$ (function of $x$ and
  $s_l$) is a nonlinear continuous function of the liquid saturation
  $s_{l}$ and is bounded for $x\in \Omega$ and $s_l\in [0,1]$. In
  addition, there exist a constant $d^*>0$ such that $\forall v\in
  \R^{d}$, $\forall x\in \Omega$, $\forall s_l\in [0,1]$,
  $\left\langle \Dlh(x,s_{l})v,v \right\rangle \geq d^* \|v\|^{2}$.
\item The function $\ga \in C^{1}\left([0,1];\R^{+} \right)$ satisfies
  $\ga(s_l)>0$ for $0<s_l< 1$ and $\ga(0)=\ga(1)=0.$ We assume that
  $\mathcal{B}^{-1}$ (the inverse of $\mathcal{B}(s_l)=\int_{0}^{s_l}\
  \ga(z) \dd z$) is an H\"{o}lder\footnote{This means that there
    exists a positive constant $b$ such that for all $a, b \in [0,\mathcal{B}(1)],$ one has
    $|\mathcal{B}^{-1}(a)-\mathcal{B}^{-1}(b)|\leq
    b|a-b|^{\theta}$.}  function of order $\theta$, with
  $0<\theta\leq 1, \text{ on } [0,\mathcal{B}(1)]$.
\end{enumerate}
Let us define the following Sobolev space
$$
H^{1}_{\Gamma_{l}}(\Omega) = \{ u\in H^{1}(\Omega); u=0 \text{ on }
\Gamma_{l} \}
$$
this is an Hilbert space with the norm $\|u
\|_{H^{1}_{\Gamma_{l}}(\Omega)} = \|\nabla u\|_{(L^{2}(\Omega))^{d}}$.

Let us state the main result of this paper


\begin{theorem}\label{theo:existence} 
  Let $(H1)$-$(H8)$ hold and let the initial conditions
  $(p^0_g,p^0_l)$ belongs in $L^2(\Omega)\times L^2(\Omega)$ with
  $0\le \sa_l^0\le 1$. Then there exists a solution
  $\left(p_g,p_l\right)$ satisfying
\begin{align}
  & p_g \in L^{2}(0,T;H^{1}_{\Gamma_l}(\Omega)) \text{ and }\sqrt{M_\alpha(\sa_\alpha)}\nabla p_\alpha\in L^2(Q_T), \label{1}\\ & \sa_l\geq 0 \text{
    a.e. in } Q_{T}, \B(\sa_l)\in L^{2}(0,T;H^{1}(\Omega)),\label{2}\\
  & \phii
  \partial_{t} (\rho_l^{h}(p_g)m(\sa_l))\in
  L^{2}(0,T;(H^{1}_{\Gamma_l}(\Omega))^{'}),\ \phii
  \partial_{t}\sa_l\in L^{2}(0,T;(H^{1}_{\Gamma_l}(\Omega))^{'}),\label{3}
  \end{align}
  in the sense that for all $\varphi, \psi \in
  C^{1}(0,T;H^{1}_{\Gamma_l}(\Omega))$ with
  $\varphi(T,.)=\psi(T,.)=0$,
  \begin{numcases}{}
  \begin{aligned}
    & - \int_{Q_T}\phii m(\sa_l)\rho_l^{h}(p_g)\partial_{t} \varphi
    \dd x \dd t - \int_{\Omega}\phii m(\sa_l^0)\rho_l^{h}(p_g^0)
    \varphi(0,x) \dd x \\ & \quad \quad +\int_{Q_{T}}\Ka
    \rho_l^{h}(p_g)M_l(\sa_l)(\nabla p_l - \rho_l(p_l)\G)\cdot \nabla
    \varphi \dd x \dd t \\ & \quad \quad\quad + \ctea\int_{Q_{T}}\Ka
    \rho_l^{h}(p_g)M_g(\sa_l)(\nabla p_g -\rho_g(p_g)\G)\cdot \nabla
    \varphi \dd x \dd t \\ & \quad \quad \quad \quad+ \int_{Q_{T}}
    \cteb \X_l^w \Dlh \nabla p_g \cdot \nabla \varphi \dd x \dd t =
    \int_{Q_{T}}r_{g}\varphi \dd x \dd t,
\end{aligned} & \label{eq:pg} \\
\begin{aligned}
  & -\int_{Q_T}\phii \sa_l \partial_{t} \psi \dd x \dd t
  -\int_{\Omega}\phii \sa_l^0 \psi(0,x) \dd x \\ & \quad \quad +
  \int_{Q_{T}}\Ka M_{l}(\sa_l) (\nabla p_l - \rho_l(p_l)\G) \cdot
  \nabla \psi \dd x \dd t =
  \int_{Q_{T}}\frac{r_{\omega}}{\rho_l^w}\psi \dd x \dd t,
\end{aligned}& \label{eq:pl}
\end{numcases}
and the initial conditions are satisfied in the sense that for all
$\xi\in H^1_{\Gamma_l}(\Omega)$, the functions
$
t\to \int_{\Omega}\phii\rho_l^{h}(p_g)m(\sa_l) \xi \dd x ,\text{ and }
t \to \int_{\Omega} \phii \sa_l \xi \dd x,
$
are in $C^0([0,T])$. Furthermore, we have
\begin{align*}
  & \left(\int_{\Omega}\phii\rho_l^h(p_g)m(\sa_l) \xi \dd
    x\right)(0)=\int_{\Omega}\phi\rho_l^h(p_g^0)m(\sa_l^0) \xi \dd
  x, \\
  & \left(\int_{\Omega}\phii \sa_l \xi \dd
    x\right)(0)=\int_{\Omega}\phii \sa_l^0 \xi \dd x.
  \end{align*}
\end{theorem}
\begin{remarque}
  Remark that, the solutions obtained in the above theorem do not satisfy that $s_l\le 1$, then the functions depend on the saturation are extended by continuity for $s_l\ge 1$.
\end{remarque}
\section{Energy estimates}
\label{sec:energy}

The notion of weak solutions is very natural provided that we explain
the origin of the requirements \eqref{1}--\eqref{3}. In this section,
we give estimates on the gradient of the global pressure and on the
gradient of the capillary term ${\cal B}$.  In order to obtain these
estimations, we define $g_g$ and ${\cal H}_g$ by
$$
g_{g}(p_{g}):=\int_{0}^{p_{g}}\frac{1}{\rho_{l}^h(z)}\dd z \text{ and
} \mathcal{H}_{g}(p_g):=\rho_l^h(p_{g})g_{g}(p_{g})-p_{g}.
$$
The function ${\cal H}_g$ verifies
$\mathcal{H}^{\prime}_{g}(p_{g})=(\rho^{h}_{l}(p_{g}))^\prime
g_{g}(p_{g})$, $\mathcal{H}_{g}(0)=0$, $\mathcal{H}_{g}(p_{g})\geq 0$
for all $p_{g}$
and $\mathcal{H}_{g}$ is sublinear with respect to $p_g$. This kind of function is introduced in \cite{MCcras, ZS10, ZS11}. \\
By multiplying \eqref{eq:hydrogene1} by $g_{g}(p_{g})$ and
\eqref{eq:eau1} by $\ctea p_l - p_g$, after integration and summation
of equations, we deduce the equality
\begin{align}
  &\notag\int_{\Omega}\phii \left[
    \partial_{t}\Big(m(\sa_l)\rho_l^{h}(p_g)\Big)g_{g}(p_g)+\dt \sa_l
    \Big(\ctea\; p_l - p_g\Big) \right]\dd x \\ &
  \label{est:est1}\qquad +
  \int_{\Omega}^{} {\bf K} M_l(\sa_l)(\grad p_l-\rho_l(p_l){\bf
    g})\cdot \grad p_l \dd x \notag \\ & \qquad \quad+ \ctea
  \int_{\Omega}^{} {\bf
    K} M_g(\sa_g)(\grad p_g-\rho_g(p_g){\bf g})\cdot \grad p_g \dd x\\
  & \notag\qquad \quad + \int_{\Omega} \cteb \X_l^w \Dlh \nabla p_g
  \cdot \nabla p_g \dd x = \int_{\Omega}\ \Big( r_{g} g_{g}(p_g) +
  \frac{r_{\omega}}{\rho_{l}^w}(\ctea p_l-p_g)\Big)\,\dd x.
\end{align}
To treat the first term of the above equality, let
\begin{align*}
  \mathcal{M} & =  \partial_{t}\Big(m(\sa_l)\rho_l^{h}(p_g)\Big)g_{g}(p_g)+\dt \sa_l \Big(\ctea\; p_l - p_g\Big) \\
  & = \partial_{t}\Big(m(\sa_l)\rho_l^{h}(p_g)g_{g}(p_g)\Big) + \dt
  \Big( \sa_l(\ctea p_l -p_g)\Big) + \ctea \sa_l \dt(p_g-p_l) - \ctea \dt p_g\\
  & = \partial_{t}\Big(m(\sa_l)\rho_l^{h}(p_g)g_{g}(p_g)\Big) + \dt
  \Big( \sa_l(\ctea p_l -p_g)\Big) + \ctea \sa_l \dt(p_c) - \ctea \dt
  p_g.
\end{align*}
Consider $\mathcal{N}$ a primitive of $\sa_l p_c^{'}(\sa_l)$. We can write $\mathcal{M}$ as $\mathcal{M} =\dt
\mathcal{E}$ where $\mathcal{E}$ is defined by
\begin{align*}
  \mathcal{E}& = m(\sa_l)\rho_l^{h}(p_g)g_{g}(p_g) + \sa_l(\ctea p_l
  -p_g) + \ctea
  \mathcal{N}(\sa_l) - \ctea p_g\\
  & = m(\sa_l)\Big(\rho_l^{h}(p_g)g_{g}(p_g) - p_g\Big) - \ctea \sa_l
  p_c(\sa_l)+ \ctea \mathcal{N}(\sa_l).
\end{align*}
From the definition of the functions $\mathcal{H}_g$ and
$\mathcal{N}$, the expression of $\mathcal{E}$ is equivalent to
$$
\mathcal{E} = m(\sa_l)\mathcal{H}_g(p_g) - \ctea \int_{0}^{\sa_l}
p_c(z)\dd z.
$$
Integrate \eqref{est:est1} over $(0,T)$, we deduce by using the
assumptions $(\textit{H1})$-$(\textit{H8})$, the positivity of
$\mathcal{H}_{g}$ and the sub-linearity of $g_{g}(p_{g})$, that
\begin{multline}\label{estimation}
  \int_{Q_{T}}M_{l} |\nabla p_{l}|^2 \dd x \dd t+ \int_{Q_{T}}M_{g}
  |\nabla p_{g}|^2 \dd x\dd t \\ + \int_{Q_T}\cteb \X_l^w \Dlh \nabla
  p_g \cdot \nabla p_g \dd x \dd t \le C\left(1 + \|p_l\|_{L^2(Q_T)}
    +\|p_g\|_{L^2(Q_T)} \right),
\end{multline}
where $C>0$ is constant. In term of global pressure, from the relation
\eqref{form}, we have the fundamental equality
\begin{equation}\label{eq_magic}
  M |\nabla p |^{2}  +
  \frac{M_{l}M_{g}}{M}|\nabla p_c |^{2}  = M_{l} |\nabla p_{l}|^2 +    M_{g} |\nabla p_{g}|^2 ,
\end{equation}
The relation \eqref{p} between the global pressure and the pressure
of each phase prove the following inequality
\begin{align*}
  \|p_l\|_{L^2(Q_T)} +\|p_g\|_{L^2(Q_T)} & \le \|p\|_{L^2(Q_T)}
  +\|\bar{p}\|_{L^2(Q_T)} + \|\tilde{p}\|_{L^2(Q_T)} \\ & \le C
  \|\nabla p\|_{L^2(Q_T)} +\|\bar{p}\|_{L^2(Q_T)}
  +\|\tilde{p}\|_{L^2(Q_T)}.
\end{align*}
The above inequality and the equality \eqref{eq_magic} combined to the
estimate \eqref{estimation} ensures that $p,\; p_g\in
L^{2}(0,T;H_{\Gamma_l}^{1}(\Omega))$ and $\mathcal{B}(s_{l})\in
L^{2}(0,T;H^{1}(\Omega))$.
\section{Construction of a regularized system}\label{sec:regular}

Before establishing Theorem \ref{theo:existence}, we introduce the
existence of regularized solutions to system
\eqref{eq:hydrogene}--\eqref{eq:eau}. First we are interested in a
non-degenerate system by adding a dissipative term on saturation
preserving the positivity of the liquid saturation. Precisely, we
consider the non-degenerate system:
\begin{numcases}{\sys_\eta}
\begin{aligned}
  & \partial_{t}\left(\phii
    m(\sa_l^{\eta})\rho_l^{h}(p_g^{\eta})\right) -
  \di(\Ka\rho_l^{h}(p_g^{\eta})M_l(\sa_l^{\eta})\left(\nabla
    p_l^{\eta}-\rho_l(p_l^\eta)\G\right))\\ & - \ctea
  \di\left(\Ka\rho_l^{h}(p_g^\eta)M_g(\sa_g^{\eta})\left(\nabla
      p_g^{\eta}-\rho_g(p_g)\G\right)\right) -\di(\cteb (\X_l^w)^\eta
  (\Dlh)^\eta\nabla p_g) \\ & + (\ctea-1)\eta\;
  \di(\rho_l^{h}(p_g)\nabla (p_g^{\eta}-p_l^{\eta}))=r_g,
  \end{aligned}\notag  \\ 
  \begin{aligned}
    \partial_{t}\left(\phii \sa_{l}^{\eta}\right) - \di\left(\Ka
      M_l(\sa_l^{\eta})\left(\nabla
        p_l^{\eta}-\rho_l(p_l^\eta)\G\right)\right)
    -\eta\;\di(\nabla(p_g^{\eta} - p_l^{\eta}))
    =\frac{r_w}{\rho_l^{w}},\notag
  \end{aligned}
\end{numcases}
completed with the initial conditions $(\ref{cd:initiale})$, and the
following mixed boundary conditions,
\begin{equation}\label{cd:bord_eta}
  \left\{ \begin{aligned}
      & p_{g}^{\eta}(t,x) = p_{l}^{\eta}(t,x)=0  & \text{ on } (0,T)\times \Gamma_{l}  \\ &
      \big( \V_l^\eta + \ctea \V_g^\eta + \cteb \X_l^w \Dlh \nabla p_g 
      + (\ctea - 1)\eta \nabla(p_{g}^{\eta} - p_{l}^{\eta}) \big)\cdot\n = 0  & \text{ on } (0,T)\times\Gamma_{n}\\ & 
      \left(\V_l^\eta - \eta \nabla(p_{g}^{\eta} 
        - p_{l}^{\eta})\right)\cdot\n = 0  & \text{ on } (0,T)\times\Gamma_{n}
    \end{aligned} \right.
\end{equation}  
where $\textbf{n}$ is the outward normal to the boundary $\Gamma_n$
and $\V_{\alpha}^\eta = -\textbf{K} M_{\alpha}(s_{\alpha}^{\eta}) (
\nabla p_{\alpha}^{\eta} -\rho_{\alpha}(p_{\alpha}^{\eta})\G)$.

  \bigskip

  We state the existence of solutions of the above system
  $(\sys_\eta)$.
\begin{theorem}\label{theo:exi_eta} Let $\text{(H1)--(H8)}$ hold.
  Let $p_{g}^{0},p_{l}^{0}\in L^2(\Omega)$, $0\le \sa_l^0(x)\le 1$.  Then, for all $\eta>0$, there exists
  $(p_{g}^{\eta},p_{l}^{\eta})$ satisfying
  \begin{align*}
    & p_\alpha^{\eta}\in L^{2}(0,T;H^{1}_{\Gamma_l}(\Omega)), \ \phii
    \partial_{t}(\rho_l^h(p_g^{\eta})m(\sa_l^{\eta}))
    \in L^{2}(0,T;(H^{1}_{\Gamma_l}(\Omega))^{'}), \\
    & \sa_l^{\eta}\geq 0 \text{ a.e. in } Q_{T},\ \sa_l^{\eta}\in
    L^{2}(0,T;H^{1}(\Omega)), \ \phii \partial_{t}\sa_l^{\eta}\in
    L^{2}(0,T;(H^{1}_{\Gamma_l}(\Omega))^{'}), \\ &
    \rho_l^h(p_g^{\eta})m(\sa_l^{\eta})\in C^0\left([0,T];L^2(\Omega)
    \right),\sa_l^\eta\in C^0\left([0,T];L^2(\Omega) \right),
\end{align*}
such that for all $\varphi, \psi \in
L^{2}(0,T;H^{1}_{\Gamma_l}(\Omega))$
\begin{equation}\label{eq:pg_eta}
\begin{aligned}
  &\left\langle \phii
    \partial_{t}(m(\sa_l^{\eta})\rho_l^{h}(p_g^{\eta})),\varphi\right\rangle
  +\int_{Q_{T}}\Ka \rho_l^{h}(p_g^\eta)M_l(\sa_l^{\eta})\left(\nabla
    p_l^{\eta}-\rho_l(p_l^\eta)\G\right)\cdot \nabla \varphi \dd x \dd t\\
  & \quad \quad + \ctea\int_{Q_{T}}\Ka
  \rho_l^{h}(p_g^{\eta})M_g(\sa_l^{\eta}) \left(\nabla p_g^{\eta}
    -\rho_g(p_g^\eta)\G\right)\cdot \nabla \varphi \dd x \dd t \\ & \quad
  \quad \quad + \int_{Q_{T}}\cteb \Xlweta (\Dlh)^\eta \nabla p_g^\eta \cdot \nabla
  \varphi \dd x \dd t + (\ctea -
  1)\eta\int_{Q_{T}}\rho_l^{h}(p_g^{\eta})\nabla (p_g^{\eta} -
  p_l^{\eta}) \cdot \nabla \varphi \dd x\dd t \\ & \quad \quad \quad
  \quad \quad= \int_{Q_{T}}r_{g}\varphi \dd x \dd t,
\end{aligned}
\end{equation}
\begin{equation}\label{eq:pl_eta}
\begin{aligned}
  & \left\langle\phii \partial_{t}\sa_l^{\eta},\psi\right\rangle +
  \int_{Q_{T}}\Ka M_{l}(\sa_l^{\eta}) \left(\nabla p_l^{\eta}
    -\rho_l(p_l^\eta)\G\right) \cdot \nabla \psi \dd x \dd t \\ & \quad
  \quad \quad \quad\quad \quad - \eta\int_{Q_{T}}\nabla (p_g^{\eta} -
  p_l^{\eta}) \cdot \nabla \psi \dd x\dd t =
  \int_{Q_{T}}\frac{r_{\omega}}{\rho_l^w}\psi \dd x \dd t,
\end{aligned}
\end{equation}
where the bracket $\left\langle .,.\right\rangle$ is the duality
product between $L^{2}(0,T;(H^{1}_{\Gamma_{l}}(\Omega))^{'})$ and
$L^{2}(0,T;H^{1}_{\Gamma_{l}}(\Omega))$.\\
\end{theorem}
For the sake of clarity, we omit the index $\eta$ in the problem
$(\mathfrak{P}_{\eta})$. The sequel of this section is devoted to the
proof of Theorem \ref{theo:exi_eta}. The existence of solution of
non-degenerated model $(\mathfrak{P}_\eta)$ is splitted in three
steps. The first one based on approached solutions solving a time
discrete system with non-degenerate mobilities.
%
For that, let be $T>0$, $M\in \N^{*}$ a number of time step, $\h=T/M$
the time step and let initialize a sequence parametrized by $\h$ with
the initial condition $p_\alpha^0$.
Then, if we consider $(p^{n,\epsilon}_g,p^{n, \epsilon}_l)\in
(L^{2}(\Omega))^{2}$ with
$\rho_l^h(p^{n,\epsilon}_g)m(s^{n,\epsilon}_l) \geq 0$ and
$s^{n,\epsilon}_l\geq 0$ at time $t_n=n\h$, we are searching for a
solution $(p^{n+1, \epsilon}_g,p^{n+1,\epsilon}_l)$ of the following system
\begin{numcases}{\sys_{\eta,\h}^\epsilon}
  \begin{aligned}
    & \phii
    \frac{m(Z(\sa^{n+1,\epsilon}_l))\rho_{l}^{h}(p^{n+1,\epsilon}_g) -
      m(\sa^{n,\epsilon}_l)\rho_{l}^{h}(p^{n,\epsilon}_g)}{\h} \\
    &\quad -
    \di\Big(\Ka\rho_{l}^{h}(p^{n+1,\epsilon}_{g})\big(M_{l}^\epsilon(\sa^{n+1,\epsilon}_l)
    \nabla p^{n+1,\epsilon}_l -
    M_{l}(\sa^{n+1,\epsilon}_l)\rho_l(p^{n+1,\epsilon}_l){\G}\big)\Big)
    \\ & \quad\quad- \ctea\;
    \di\Big(\Ka\rho_{l}^{h}(p^{n+1,\epsilon}_g)\big(M_{g}^\epsilon(\sa^{n+1,\epsilon}_g)\nabla
    p^{n+1,\epsilon}_g-M_{g}(\sa^{n+1,\epsilon}_g)\rho_g(p^{n+1,\epsilon}_g){\G}\big)\Big)
    \\ & \quad\quad \quad+ (\ctea - 1)\eta\;
    \di\Big(\rho_{l}^{h}(p^{n+1,\epsilon}_g)\nabla
    (p^{n+1,\epsilon}_g-p^{n+1,\epsilon}_l)\Big)\\ &
    \quad\quad\quad\quad -\di\Big(\cteb\Xlwetan \Dlh \nabla
    p^{n+1,\epsilon}_g\Big) =r^{n+1}_g,
\end{aligned} & \label{sy\sa_pg_h} \\
\begin{aligned}
  & \phii \frac{Z(\sa^{n+1,\epsilon}_l) - \sa_l^{n,\epsilon}}{\h} -
  \di\Big(\Ka \big( M_{l}^\epsilon(\sa^{n+1,\epsilon}_l)\nabla
  p^{n+1,\epsilon}_l-
  M_{l}(\sa^{n+1,\epsilon}_l)\rho_l(p_l^{n+1,\epsilon}){\G}\big)\Big)
  \\ & \quad \quad \quad\quad \quad \quad \quad \quad\quad \quad
  -\eta\; \di\Big(\nabla(p^{n+1,\epsilon}_g - p^{n+1,\epsilon}_l)\Big)
  =\frac{r_w^{n+1}}{\rho_l^{w}},
\end{aligned} & \label{sy\sa_pl_h}    
\end{numcases}
where $M_\alpha^\epsilon = M_\alpha + \epsilon$, with $\epsilon >0$,
with the boundary conditions \eqref{cd:bord_eta}. The regularization
of the mobilities lead to the loss of the positivity on the liquid
saturation. So, the functions $M_\alpha$ and $Z$ are extended on $\R$
by continuity outside $[0,1]$.

This technique of semi-discretization method in time has been used by
Alt and Luckhaus \cite{Alt83} for degenerate parabolic system and has
been employed in \cite{CS07,ZS10,CS10} for a porous medium. A Leray-Schauder's fixed
point theorem \cite{Zeidler} allows to define a solution
$(p_g^{n+1,\epsilon},p_l^{n+1,\epsilon})$ for the system
$\sys_{\eta,\h}^\epsilon$.

\bigskip

The second step is devoted to pass to the limit when $\epsilon$ goes
to zero and to prove the positivity of the liquid pressure. A uniform
estimate (with respect to $\epsilon$) based on the scalar product of
\eqref{sys_pg_h} with $g_{g}(p_{g}) :=
\int_{0}^{p_{g}}\frac{1}{\rho_{l}^{h}(z)}\dd z$ and \eqref{sys_pl_h}
with $\psi = \ctea p_l-p_g$ ensures by using \eqref{p} and
\eqref{eq_magic} that
  \begin{align*}
    &(p^{\epsilon})_{\epsilon}, (p_{\alpha}^{\epsilon})_{\epsilon}               &&\text{is uniformly bounded in $H^{1}_{\Gamma_{l}}(\Omega)$},\\
    &(\B(\sa_{l}^{\epsilon}))_{\epsilon}    &&\text{is uniformly bounded in $H^{1}(\Omega)$},\\
    &(\nabla p_c(\sa_{l}^{\epsilon}))_{\epsilon} &&\text{is uniformly
      bounded in $L^{2}(\Omega)$}.
    \end{align*}
     Up to a subsequence, the sequences
     $(s_{\alpha}^{\epsilon})_{\epsilon},(p^{\epsilon})_{\epsilon},(p_{\alpha}^{\epsilon})_{\epsilon}$,
     verify the following convergences
 \begin{align*}
   &   p^{\epsilon} {\longrightarrow} p  \text{ and } p_{\alpha}^{\epsilon} \longrightarrow p_{\alpha}                                 & &   \quad  \text{weakly in $H^{1}_{\Gamma_l}(\Omega)$ and a.e. in $\Omega$},\\
   &   \mathcal{B}(s_{l}^{\epsilon}){\longrightarrow} \mathcal{B}(s_{l})           &&   \quad  \text{weakly in $H^{1}(\Omega)$ and a.e.  in $\Omega$},\\
   &   Z(s_{l}^{\epsilon}) \longrightarrow Z(s_{l})                  & &   \quad  \text{strongly in $L^{2}(\Omega)$ and a.e. in $\Omega$}.
 \end{align*}
 Then, pass to the limit as $\epsilon$ goes to zero in formulations
 \eqref{sys_pg_h}-\eqref{sys_pl_h} to obtain $(p_g,p_l)\in
 H^1_{\Gamma_1}(\Omega)\times H^1_{\Gamma_1}(\Omega)$ solution of
\begin{equation}\label{pg_hh}
\begin{aligned}
  &\int_{\Omega}\frac{m(Z(\sa_l))\rho_l^{h}(p_g) -
    \rho^{*}m(\sa_l^{*})}{\h}\varphi \dd x  + \int_{\Omega}\cteb
  \X_l^{w} \Dlh \nabla p_g \cdot \nabla \varphi \dd x \\ &
  +\int_{\Omega}\Ka \rho_l^{h}(p_g)M_l(\sa_l)\left(\nabla
    p_l-\rho_l(p_l)\G\right)\cdot \nabla \varphi \dd x \\ & \quad
  \quad \quad \quad + \ctea \int_{\Omega}\Ka \rho_l^{h}(p_g)
  M_g(\sa_g)\left( \nabla p_g -\rho_g(p_g)\G\right)\cdot \nabla
  \varphi \dd x \\ & \quad \quad \quad \quad \quad \quad +
  (\ctea - 1)\eta\int_{\Omega}\rho_l^{h}(p_g)\nabla (p_g - p_l) \cdot
  \nabla \varphi \dd x = \int_{\Omega}r_{g}\varphi \dd x ,
\end{aligned}
\end{equation}
\begin{equation}\label{pl_hh}
\begin{aligned}
  &\int_{\Omega}\frac{Z(\sa_{l}) - \sa_l^{*}}{\h}\psi\dd x  +
  \int_{\Omega}\Ka M_{l}(\sa_l) \left( \nabla p_l-\rho_l(p_l)\G\right)
  \cdot \nabla \psi \dd x \\\ & \quad \quad \quad \quad \quad\quad
  \quad \quad \quad \quad - \eta\int_{\Omega}\nabla (p_g - p_l) \cdot
  \nabla \psi \dd x =
  \int_{\Omega}\frac{r_{\omega}}{\rho_l^w}\psi \dd x,
\end{aligned}  
\end{equation}
pour tout $\left(\varphi, \psi\right) \in
(H^{1}_{\Gamma_l}(\Omega))^2$.\\

To prove, that the liquid saturation is positive, we consider $\psi =
(\sa_l)^-$ in \eqref{pl_hh} and according to the extension of the
mobility of each phase ($M_l(s_l)s_l^{-}=0$) and 
the fact that ($Z(s_l)s_l^{-}=0$), we deduce that $s_l\ge 0$.\\


The third step is devoted to pass to the limit as $\h$ goes to zero to
prove the existence of a solution of the problem
$(\mathfrak{P}_\eta)$. For this, we will show some uniform estimates
with respect to $\h$ to obtain uniformly bounded on some quantities.

The next lemma gives us some uniform estimates with respect to $\h$.
\begin{lemme}\label{lem1}
  $\text{(Uniform estimates with respect to $\h$)}$ The solution of
  $(\ref{sy\sa_pg_h})-(\ref{sy\sa_pl_h})$ satisfies
  \begin{equation}\label{estimation_inde_h}
  \begin{aligned}
    &\frac{1}{\h}\int_{\Omega}\phii \left(\
      m(\sa^{n+1}_{l})\mathcal{H}_g(p^{n+1}_{g}) -
      m(\sa^{n}_{l})\mathcal{H}_g(p^{n}_{g})\right)\dd x -
    \frac{1}{\h} \int_{\Omega} \phii \left( \pc(\sa^{n+1}_{l}) -
      \pc(\sa^{n}_{l}) \right)\ \dd x \\ & \quad \quad + \ctea
    k_0\int_{\Omega} M_l(\sa^{n+1}_{l})\nabla p^{n+1}_{l} \cdot \nabla
    p^{n+1}_{l} \dd x + \ctea k_0\int_{\Omega}
    M_g(\sa^{n+1}_{g})\nabla p^{n+1}_{g} \cdot \nabla p^{n+1}_{g} \dd
    x \\ & \quad \quad \quad+ C_1 \int_{\Omega} \nabla p^{n+1}_{g}
    \cdot \nabla p^{n+1}_{g} \dd x + \ctea
    \eta\int_{\Omega}|\nabla (p_{g}^{n+1} - p_{l}^{n+1})|^2\dd x\\
    & \quad \quad \quad \quad \leq C_2 \left(\| r_g^{n+1}
      \|^2_{L^2(\Omega)} + \| r_\omega^{n+1} \|^2_{L^2(\Omega)}
    \right),
\end{aligned}
\end{equation}
where C does not depend on $\h$. The functions ${\cal H}$ and $\pc$
are defined by 
$$
\mathcal{H}_{g}(p_{g}) := \rho_{l}^{h}(p_{g})g_{g}(p_{g}) - p_{g}, \;
\pc(\sa_l) := \int_{0}^{\sa_l}p_c(z)\dd z \text{ et } g_{g}(p_{g}) =
\int_{0}^{p_{g}}\frac{1}{\rho_{l}^{h}(z)}\dd z.
$$
\end{lemme}
\begin{proof} Let forget the exponent $n+1$ in this proof and let denote
  with the exponent $*$ the physical quantities at time $t_n$.  So,
  since $g_{g}$ is concave $(g^{''}_g(p)\leq 0)$, we have
$$
g_{g}(p_{g}) \leq g_{g}(p^{*}_{g}) + g^{'}_{g}(p^{*}_{g})(p_{g} -
p^{*}_{g}),
$$
and from the definition of $\mathcal{H}_{g}$, one gets
  \begin{equation}\label{00}
    \begin{aligned}
      \left( \rho_{l}^{h}(p_{g})m(\sa_{l}) -
        \rho_{l}^{h}(p^{*}_{g})m(\sa^{*}_{l}) \right)g_{g}(p_{g}) +
      \left(\sa_{l} - \sa_l^{*}\right)(\ctea p_l-p_g) \\ \geq
      \mathcal{H}_{g}(p_{g})m(\sa_{l}) -
      \mathcal{H}_{g}(p^{*}_{g})m(\sa^{*}_{l}) - \ctea(s_{l} -
      s^{*}_{l})p_c(s_{l}).
    \end{aligned}
  \end{equation} 
  Using the concavity of $ \pc$ we have the inequality : $(s_{l} -
  s^{*}_{l})p_c(s_{l})\leq \pc(s_{l}) - \pc(s^{*}_{l})$, and the above
  inequality $(\ref{00})$, we obtain
  the following inequality
\begin{equation}\label{132}
    \begin{aligned}
      \left( \rho_{l}^{h}(p_{g})m(\sa_{l}) -
        \rho_{l}^{h}(p^{*}_{g})m(\sa^{*}_{l}) \right)g_{g}(p_{g}) +
      \left(\sa_{l} - \sa_l^{*}\right)(\ctea p_l-p_g) \\ \geq
      \mathcal{H}_{g}(p_{g})m(\sa_{l}) -
      \mathcal{H}_{g}(p^{*}_{g})m(\sa^{*}_{l}) - \ctea \pc(\sa_{l}) +
      \ctea \pc(\sa^{*}_{l}).
    \end{aligned}
  \end{equation} 
  Now, to obtain the inequality $(\ref{estimation_inde_h})$, we just
  have to multiply \eqref{pg_hh} by $g_{g}(p^{n+1}_{g})$ and
  \eqref{pl_hh} by $(\ctea p^{n+1}_{l} - p^{n+1}_{g})$, sum this two
  equations and use the inequality \eqref{132}.
\end{proof}

The limit as $\h$ goes to zero is similar to the limit as
$\epsilon$ goes to zero with additional difficulties on time
derivative terms which are overcome in the same manner as in
\cite{CS07,ZS10,CS10}.

\section{Existence of solutions of the degenerate system}

We have shown in the previous section \ref{sec:regular}, the existence
of a solution $(p_g^\eta,p_l^\eta)$ of the problem $\sys_\eta$. The
aim of this section is to pass to the limit as $\eta$ goes to the zero
to prove the main result of this paper. 

The first point to do this is to obtain uniform energy estimates with respect
to $\eta$. 
The second point is devoted to gat uniform estimates o, space and time translates which provide compactness results on solution by virtue of Kolmogorov's theorem. 
Next, 
we will be able to pass to the limit as $\eta$ goes to zero. \\

Now, we state the following two lemmas in order to establish uniform estimates with
respect to $\eta$.

\begin{lemme}\label{16}
  The sequences $(\sa^{\eta}_\alpha)_{\eta}$ and $(p^{\eta})_{\eta}$ satisfy
\begin{align} 
  &\sa^{\eta}_{l}\geq 0                                 &&     \text{ almost everywhere  in } Q_{T}, \label{1eta}\\
  &(p^{\eta})_{\eta},\; (p_g^\eta)_\eta                                         &&    \text{ is uniformly bounded in $L^{2}(0,T;H^{1}_{\Gamma_{l}}(\Omega))$}, \label{3eta}\\
  &(\sqrt{\eta}\  \nabla p_c(\sa^{\eta}_{l}))_{\eta}              &&    \text{ is uniformly bounded in  $L^{2}(Q_{T})$} ,\label{4eta}\\
  &(\sqrt{M_\alpha(\sa^{\eta}_\alpha)}\ \nabla p^{\eta}_\alpha)_{\eta}    &&    \text{ is uniformly bounded in  $L^{2}(Q_{T})$} ,\ \alpha=l,g\label{5eta}\\
  &(\B(\sa^{\eta}_{l}))_{\eta}                                  &&    \text{ is uniformly bounded in  $L^{2}(0,T;H^{1}(\Omega))$}, \label{7eta}\\
  &(\phii \partial_{t}(\rho_{l}^{h}(p^{\eta}_{g})m(\sa^{\eta}_{l})))_{\eta} &&     \text{ is uniformly bounded in  $L^{2}(0,T;(H^{1}_{\Gamma_{l}}(\Omega))^{'})$}, \label{8eta}\\
  &(\phii \partial_{t}(\sa^{\eta}_{l}))_{\eta} && \text{ is uniformly
    bounded in $L^{2}(0,T;(H^{1}_{\Gamma_{l}}(\Omega))^{'})$}.
  \label{9eta}
\end{align}
\end{lemme}
\begin{proof} The positivity of the saturation \eqref{1eta} is
  conserved through the limit process. For the next four estimates, we
  just have to multiply \eqref{eq:pg_eta} by $g_{g}(p^{\eta}_{g}) =
  \int_{0}^{p^{\eta}_{g}}\frac{1}{\rho_{l}^{h}(z)\dd z}$ and
  \eqref{eq:pl_eta} by $\ctea p_l^\eta-p_g^\eta$ and adding them. 
  We follow the same calculation as in  section \ref{sec:energy} to provide 
 the energy estimates \eqref{3eta}--\eqref{7eta}.\\

  For all $\varphi,\ \psi\in L^{2}(0,T;H^{1}_{\Gamma_{l}}(\Omega))$
  and by using the formulation \eqref{eq:pg_eta}--\eqref{eq:pl_eta}
  with the relation \eqref{p} between the pressure of each phase and
  the global pressure, one gets
  \begin{multline*}
     \Big|\left\langle \phii
       \partial_{t}(\rho_{l}^{h}(p_{g}^{\eta})m(\sa_{l}^{\eta})),\varphi\right\rangle
     \Big|\leq  \Big|(\ctea-1)\eta
     \int_{Q_{T}}\rho_{l}^{h}(p^{\eta}_{g})\nabla
     p_c(\sa^{\eta}_{l})\cdot \nabla \varphi\dd x \dd t \Big| \\  +
     \Big|\int_{Q_{T}}\Ka
     \rho_{l}^{h}(p^{\eta}_{g})\left(M_{l}(\sa^{\eta}_{l})\nabla
       p^{\eta} + \nabla \B(\sa^{\eta}_{l}) \right)\cdot \nabla
     \varphi \dd x \dd t \Big|\\  + \Big|\int_{Q_{T}}X_l^{w}
     D_l^{h}(\rho_l^{h}(p_g^{\eta})^{'})\nabla p_g \cdot \nabla
     \varphi \dd x \dd t  \Big| + \Big| \int_{Q_{T}}r_g\varphi \dd x \dd t \Big|,
  \end{multline*}
and
\begin{align*}
    \Big|\left\langle \phii
      \partial_{t}(\sa_{l}^{\eta},\psi\right\rangle \Big| \leq &
    \Big|\ctea \eta\int_{Q_{T}}\nabla p_c(\sa^{\eta}_{l})\cdot \nabla
    \psi\dd x \dd t \Big| \\ & + \Big|\int_{Q_{T}}\Ka
    \left(M_{l}(\sa^{\eta}_{l})\nabla p^{\eta} + \nabla
      \B(\sa^{\eta}_{l}) \right)\cdot \nabla \psi \dd x \dd t \Big|\\
    & + \Big| \int_{Q_{T}}\frac{r_w}{\rho_w}\psi \dd x \dd t \Big|,
  \end{align*}
where the bracket $\left\langle \cdot,\cdot\right\rangle$ represents
the duality product between
$L^{2}(0,T;(H^{1}_{\Gamma_{l}}(\Omega))^{'})$ and
$L^{2}(0,T;H^{1}_{\Gamma_{l}}(\Omega))$.

From the estimations $(\ref{3eta})$ and $(\ref{7eta})$, we deduce
$$
|\left\langle \phii
  \partial_{t}(\rho_{l}^{h}(p_{g}^{\eta})m(\sa_{l}^{\eta})),\varphi\right\rangle
|\leq C\|\varphi \|_{ L^{2}(0,T;H^{1}_{\Gamma_{l}}(\Omega))},
$$
and
$$
|\left\langle \phii \partial_{t}(\sa_{l}^{\eta}),\varphi\right\rangle
|\leq C\|\psi \|_{ L^{2}(0,T;H^{1}_{\Gamma_{l}}(\Omega))},
$$
which establishes \eqref{8eta}--\eqref{9eta} and proves Lemma \ref{16}.
\end{proof}

In the next lemma, we derive estimates on differences of space and
time translates of the function $U^\eta=\rho_l^h(p_g^\eta) m(s_l^\eta)$
which imply that the sequence $(\rho_l^h(p_g^\eta)m(s_l^\eta))_\eta$ is
relatively compact in $L^1(Q_T)$.
\begin{lemme}\label{time-and-space-translate}
  $\left(\text{Space and time translate of } U\right)$. Under the
  assumptions $({H}1)-({H}8)$,  the following inequalities hold :
  \begin{align}
    & \int_{\Omega^{'}\times (0,T)}|U^\eta(t,x+y)- U^\eta(t,x)| \dd x \dd t \le 
    \omega(|y|), \label{est:space}\\ &
    \iint_{\Omega \times
      (0,T-\tau)}|U^\eta(t+\tau,x) - U^\eta(t,x)| \dd x \dd t \le \tilde{\omega}(\tau),\label{est:time}
\end{align}
for all $y \in \R^3$ and  for all $\tau\in (0,T)$; with $\Omega '=\{x \in \Omega, x+y\in
\Omega \}$ and the function $\omega$ and $\tilde{\omega}$ are continuous, independent of $\eta$ and satisfying $\lim_{|y|\to 0}\omega(|y|)=0$ and 
 $\lim_{\tau\to  0}\tilde{\omega}(\tau)=0$.
\end{lemme}
\begin{proof}
  For the space translates, we observe that 
\begin{equation*}
      \begin{split}
        &\int_{(0,T)\times\Omega^{'}} |U^\eta(t,x+y)-U^\eta(t,x)| \dd x\dd t\\
        & = \int_{(0,T)\times\Omega^{'}}
        \Big|\Big(\rho_l^h(p_g^\eta)m(s_l^\eta)\Big)(t,x+y) -
        \Big(\rho_l^h(p_g^\eta)m(s_l^\eta)\Big)(t,x)\Big| \dd x\dd t\\
        & \le \int_{(0,T)\times\Omega^{'}}
        \Big|m(s_l^\eta)(t,x+y) \big(\rho_l^h(p_g^\eta(t,x+y)) -\rho_l^h(p_g^\eta(t,x))\big) \Big|\dd x\dd t \\
        & \qquad +   \int_{(0,T)\times\Omega^{'}} \Big| \rho_l^h(p_g^\eta)(t,x)(m(s_l^\eta)(t,x+y)  - m(s_l^\eta) (t,x)) \Big|\dd x\dd t\\
        & \le \mathcal{E}_1+\mathcal{E}_2,
        \end{split}
\end{equation*}
where $\mathcal{E}_1$ and $\mathcal{E}_2$ defined as follows
\begin{equation}\label{e1trans}
  \mathcal{E}_1 = \rho_M \int_{(0,T)\times\Omega^{'}} \Big|s_l^\eta(t,x+y)  - s_l^\eta(t,x)\Big|\dd x\dd t,
\end{equation}
\begin{equation}\label{e2transs}
  \mathcal{E}_2 = \int_{(0,T)\times\Omega^{'}}
  \Big|\rho_l^h(p_g^\eta(t,x+y)) -\rho_l^h(p_g^\eta(t,x)) \Big|\dd x\dd t.
\end{equation}

To handle with the space translates on saturation, we use the fact
that $\B^{-1}$ is an H\"older function, applying the
Cauchy-Schwarz inequality and from \eqref{7eta}, we deduce
\begin{align}\label{E1}
  \mathcal{E}_1 & \le C \Big[\int_{(0,T)\times\Omega^{'}}
  \Big|\B(s_l^\eta(t,x+y) )- \B(s_l^\eta(t,x))\Big| \dd x\dd
  t\Big]^\theta \notag \\ & \le C \Big[\int_{0}^{T}\int_{\Omega^{'}}
  \Big(
  \int_{0}^{1} \nabla \B(s_l^\eta(t,x+ry)).y\dd r\Big)\dd x \dd t\Big]^\theta \\
  & \le C \Big[\int_{0}^{T}\int_{\Omega^{'}} \Big( \int_{0}^{1}
  \big|\nabla \B(s_l^\eta(t,x+ry))\big|^2\dd r\Big)^\frac{1}{2}  |y| \dd x \dd t\Big]^\theta
  \notag \\ & \le C |y|^\theta\notag.
\end{align}

To treat the space translates of $\mathcal{E}_2$, we use the relationship between
the gas pressure and the global pressure, namely : $p_g=p-\tilde{p}$
defined in \eqref{p}, then, from the estimation on the global
pressure \eqref{3eta} and the estimate \eqref{E1} we have
$$
\mathcal{E}_2 \le C (|y|+|y|^\theta). 
$$
Define $V^\eta = \phii U^\eta$. From assumption $(H1)$ on the porosity, we deduce thee space  translates  on $V^\eta$.  
 The proof of the time translates of $V^\eta$ can be found in
\cite{eymard-hilhorst} for more details.
\end{proof}

From the previous two lemmas, we deduce the following convergences.
\begin{lemme} (Strong and weak convergences). Up to a subsequence the
  sequence $(\sa^{\eta}_\alpha)_{\eta}$,
  $(p^{\eta}:=p^{\eta}_{l}+\overline{p}(\sa^{\eta}_{l}))_{\eta}$ and
  $(p_\alpha^{\eta})_{\eta}$ verify the following convergence
\begin{align}
  & p^{\eta}\longrightarrow p && \text{ weakly in} \ L^{2}(0,T;H^{1}_{\Gamma_l}(\Omega)) ,\label{1ta}\\
  & \B(\sa_{l}^{\eta})\longrightarrow \B(\sa_{l}) && \text{ weakly in} \ L^{2}(0,T;H^{1}(\Omega)) ,\label{2ta}\\
  & p_g^{\eta}\longrightarrow p_g && \text{ weakly in} \ L^{2}(0,T;H^{1}_{\Gamma_l}(\Omega)) ,\label{11:59}\\
  & \sa_{l}^{\eta}\longrightarrow \sa_{l} &&  \text{ strongly in } L^2(Q_T) \text{, a.e. in} \ Q_T,\label{4ta}\\
  & \sa_l\geq 0  && \text{ almost everywhere in } Q_T, \label{6ta}\\
  & p_\alpha^{\eta}\longrightarrow p_\alpha && \text{ almost everywhere in } \ Q_T ,\label{7ta}\\
  & \phii \partial_t (U^\eta)\longrightarrow \phii \partial_t (\rho_l^{h}(p_g)m(\sa_l))&&  \text{ weakly in} \ L^{2}(0,T;H^{1}_{\Gamma_l}(\Omega)) ,\label{8ta}\\
  & \phii \partial_t \sa_l^{\eta}\longrightarrow \phii \partial_t
  \sa_l && \text{ weakly in } \
  L^{2}(0,T;H^{1}_{\Gamma_l}(\Omega)),\label{9ta}
\end{align}
where $U^\eta=\rho_l^{h}(p_g^{\eta})m(\sa_l^{\eta})$.
\end{lemme}              
\begin{proof} The weak convergences \eqref{1ta}--\eqref{11:59} follows
  from the uniform estimates \eqref{3eta} and \eqref{7eta} of lemma \ref{16}.\\
  By the Riesz-Frechet-Kolmogorov compactness criterion, the relative
  compactness of $V^\eta$ in $L^1(Q_T)$ is a consequence of Lemma
  \ref{time-and-space-translate} and then ensures the following strong
  convergences 
\begin{align*}
  & \phi\rho_l^h(p_g^\eta)m(\sa_l^\eta) \longrightarrow l \text{ strongly
    in $L^1(Q_T)$ and a.e. in $Q_T$ },
\end{align*} 
and consequently 
\begin{align}\label{conv-}
  & \rho_l^h(p_g^\eta)m(\sa_l^\eta) \longrightarrow U =l/\phii\text{ strongly
    in $L^1(Q_T)$ and a.e. in $Q_T$ },
\end{align} 
In order to prove the convergence \eqref{4ta}, we reproduce the
previous Lemma \ref{time-and-space-translate} for $V^\eta = \phii s_l^\eta$ and
as an application of the Riesz-Frechet-Kolmogorov compactness
criterion we establish \eqref{4ta}. And from \eqref{4ta}, we deduce
\eqref{6ta}.

The convergence \eqref{conv-} combined with \eqref{4ta}, to prove that
\begin{align*}
  & p_g^\eta \longrightarrow p_g \text{ a.e. in $Q_T$ },
\end{align*} 
then the convergence $(\ref{7ta})$ for $\alpha=g$ is then established.
And again as consequence of \eqref{4ta} with the capillary pressure
law, we deduce $(\ref{7ta})$ for $\alpha=l$. At last, the weak
convergence \eqref{8ta} and \eqref{9ta} is a
consequence of the estimate $(\ref{8eta})$ and $(\ref{9eta})$.\\
 \end{proof}
 In order to achieve the proof of Theorem $\ref{theo:existence}$, it
 remains to pass to the limit as $\eta$ goes to zero in the
 formulations \eqref{eq:pg_eta}--\eqref{eq:pl_eta}, for all smooth
 test functions $\varphi$ and $\psi$ in
 $C^1([0,T];H^1_{\Gamma_l}(\Omega))$ such that $\varphi(T)=\psi(T)=0$,
\begin{equation}\label{fv:eta_pg}
  \begin{aligned}
    & -\int_{Q_T}\phii
    \rho_{l}^h(p_{g}^{\eta})m(\sa_{l}^{\eta})\partial_{t}\varphi \dd x
    \dd t + \int_{Q_{T}}\cteb \X_l^{w} \Dlh \nabla p_g \cdot \nabla
    \varphi \dd x \dd t \\ & \quad \quad + \int_{Q_{T}}\Ka
    \rho_l^{h}(p_g^\eta)M_l(\sa_l^{\eta})\left(\nabla
      p_l^{\eta}-\rho_l(p_l)\G\right)\cdot \nabla \varphi \dd x \dd t\\
    & \quad \quad \quad \quad + \ctea\int_{Q_{T}}\Ka
    \rho_l^{h}(p_g^{\eta})M_g(\sa_l^{\eta}) \left(\nabla p_g^{\eta}
      -\rho_g(p_g)\G\right)\cdot \nabla \varphi \dd x \dd t \\ &\quad
    \quad \quad \quad \quad \quad \quad \quad +
    (\ctea-1)\eta\int_{Q_{T}}\rho_l^{h}(p_g^{\eta})\nabla (p_g^{\eta}
    - p_l^{\eta}) \cdot \nabla \varphi \dd x\dd t \\ &\quad \quad
    \quad \quad \quad \quad \quad \quad \quad \quad =
    \int_{Q_{T}}r_{g}\varphi \dd x \dd t+ \int_{Q_T}\phii
    \rho_{l}^h(p_{g}^{0})m(\sa_{l}^{0})\varphi(0,x) \dd x \dd t,
\end{aligned}
\end{equation}
\begin{equation}\label{fv:eta_pl}
  \begin{aligned}
    & -\int_{Q_T}\phii \sa_{l}^{\eta}\partial_{t}\psi \dd x \dd t +
    \int_{Q_{T}}\Ka M_{l}(\sa_l^{\eta}) \left(\nabla p_l^{\eta}
      -\rho_l(p_l)\G\right) \cdot \nabla \psi \dd x \dd t\\ & \quad
    \quad - \eta\int_{Q_{T}}\nabla (p_g^{\eta} - p_l^{\eta}) \cdot
    \nabla \psi \dd x\dd t =
    \int_{Q_{T}}\frac{r_{\omega}}{\rho_l^w}\psi \dd x \dd t +
    \int_{Q_T}\phii \sa_{l}^{0}\psi(0,x) \dd x \dd t,
\end{aligned}
\end{equation}

The first term in \eqref{fv:eta_pg} and \eqref{fv:eta_pl} converge due
to the strong convergence of $\rho_l^h(p_g^\eta)m(\sa_l^\eta)$ to
$\rho_l^h(p_g)m(\sa_l)$ in $L^2(Q_T)$ and the strong convergence of
$\sa_l^\eta$ to $\sa_l$ in $L^2(Q_T)$.

The third and fourth term in \eqref{fv:eta_pg} can be written as,
\begin{equation}\label{ya}
  \begin{aligned}
    \int_{Q_T}\Ka M_\alpha(\sa_\alpha^\eta)
    \rho_\alpha(p_\alpha^\eta)\nabla p_\alpha^\eta \cdot \nabla
    \varphi \dd x \dd t & = \int_{Q_T}\Ka M_\alpha(\sa_\alpha^\eta)
    \rho_\alpha(p_\alpha^\eta)\nabla p^\eta \cdot \nabla \varphi \dd x
    \dd t \\ & + \int_{Q_T}\Ka \rho_\alpha(p_\alpha^\eta)\nabla
    \B(\sa_\alpha^\eta) \cdot \nabla \varphi \dd x \dd t.
  \end{aligned}
\end{equation}

The two terms on the right hand side of the equation $(\ref{ya})$
converge arguing in two steps. Firstly, the Lebsgue theorem and the
convergences \eqref{4ta} and \eqref{7ta}, establish
\begin{align*}
  & \rho_\alpha(p_\alpha^\eta) M_\alpha(\sa_\alpha^\eta)\nabla \varphi
  \longrightarrow \rho_\alpha(p_\alpha)M_\alpha(\sa_\alpha)\nabla
  \varphi && \text{ strongly in } (L^2(Q_{T}))^d,\\ &
  \rho_\alpha(p_\alpha^\eta) \nabla \varphi \longrightarrow
  \rho_\alpha(p_\alpha)\nabla \varphi && \text{ strongly in }
  (L^2(Q_{T}))^d.
\end{align*}
Secondly, the weak convergence on global pressure \eqref{1ta} and the
weak convergence \eqref{2ta} combined to the above strong convergences
allow the convergence for the terms of the right hand side of
\eqref{ya}. In the same way for the second term of the equation \eqref{fv:eta_pl}.\\
The fifth term of equations $(\ref{fv:eta_pg})$ can be written as
$$
\eta\int_{Q_T}\rho_l^h(p_g^\eta)\nabla (p_g^\eta - p_l^\eta) \nabla
\varphi \dd x \dd t = \sqrt{\eta} \int_{Q_T}\rho_l^h(p_g^\eta)
(\sqrt{\eta}\nabla p_c(\sa_l^\eta))\nabla \varphi \dd x \dd t,
$$
the Cauchy-Schwarz inequality and the uniform estimate $(\ref{4eta})$
ensure the convergence of this term to zero as $\eta$ goes to zero. In
the same way for the third term of equation \eqref{fv:eta_pl}. The
other terms converge using \eqref{4ta}--\eqref{7ta} and the
Lebesgue dominated convergence theorem.\\
The weak formulations \eqref{eq:pg} and \eqref{eq:pl} are then
established. And The main theorem \ref{theo:existence} is then
established.



\end{document}